\documentclass[12pt]{amsart}
\usepackage{cases}
\usepackage{mathrsfs}
\usepackage{amsfonts}
\usepackage{amscd}

\usepackage{latexsym,amssymb}
\usepackage{a4wide}
\usepackage{epic,eepic,latexsym, amssymb, amscd, amsfonts, xypic, floatflt}
\usepackage{amsthm}
\usepackage{amsmath}
\usepackage{amscd}
\usepackage{graphicx}
\usepackage[mathscr]{eucal}
\input xy
\xyoption{all}
\usepackage{verbatim}

\numberwithin{equation}{section}
\begin{document}
\theoremstyle{plain}
\newtheorem{thm}{Theorem}[section]
\newtheorem{theorem}[thm]{Theorem}
\newtheorem{lemma}[thm]{Lemma}
\newtheorem{corollary}[thm]{Corollary}
\newtheorem{proposition}[thm]{Proposition}
\newtheorem{conjecture}[thm]{Conjecture}
\theoremstyle{definition}
\newtheorem{remark}[thm]{Remark}
\newtheorem{remarks}[thm]{Remarks}
\newtheorem{definition}[thm]{Definition}
\newtheorem{example}[thm]{Example}

\newcommand{\caA}{{\mathcal A}}
\newcommand{\caB}{{\mathcal B}}
\newcommand{\caC}{{\mathcal C}}
\newcommand{\caD}{{\mathcal D}}
\newcommand{\caE}{{\mathcal E}}
\newcommand{\caF}{{\mathcal F}}
\newcommand{\caG}{{\mathcal G}}
\newcommand{\caH}{{\mathcal H}}
\newcommand{\caI}{{\mathcal I}}
\newcommand{\caJ}{{\mathcal J}}
\newcommand{\caK}{{\mathcal K}}
\newcommand{\caL}{{\mathcal L}}
\newcommand{\caM}{{\mathcal M}}
\newcommand{\caN}{{\mathcal N}}
\newcommand{\caO}{{\mathcal O}}
\newcommand{\caP}{{\mathcal P}}
\newcommand{\caQ}{{\mathcal Q}}
\newcommand{\caR}{{\mathcal R}}
\newcommand{\caS}{{\mathcal S}}
\newcommand{\caT}{{\mathcal T}}
\newcommand{\caU}{{\mathcal U}}
\newcommand{\caV}{{\mathcal V}}
\newcommand{\caW}{{\mathcal W}}
\newcommand{\caX}{{\mathcal X}}
\newcommand{\caY}{{\mathcal Y}}
\newcommand{\caZ}{{\mathcal Z}}
\newcommand{\fA}{{\mathfrak A}}
\newcommand{\fB}{{\mathfrak B}}
\newcommand{\fC}{{\mathfrak C}}
\newcommand{\fD}{{\mathfrak D}}
\newcommand{\fE}{{\mathfrak E}}
\newcommand{\fF}{{\mathfrak F}}
\newcommand{\fG}{{\mathfrak G}}
\newcommand{\fH}{{\mathfrak H}}
\newcommand{\fI}{{\mathfrak I}}
\newcommand{\fJ}{{\mathfrak J}}
\newcommand{\fK}{{\mathfrak K}}
\newcommand{\fL}{{\mathfrak L}}
\newcommand{\fM}{{\mathfrak M}}
\newcommand{\fN}{{\mathfrak N}}
\newcommand{\fO}{{\mathfrak O}}
\newcommand{\fP}{{\mathfrak P}}
\newcommand{\fQ}{{\mathfrak Q}}
\newcommand{\fR}{{\mathfrak R}}
\newcommand{\fS}{{\mathfrak S}}
\newcommand{\fT}{{\mathfrak T}}
\newcommand{\fU}{{\mathfrak U}}
\newcommand{\fV}{{\mathfrak V}}
\newcommand{\fW}{{\mathfrak W}}
\newcommand{\fX}{{\mathfrak X}}
\newcommand{\fY}{{\mathfrak Y}}
\newcommand{\fZ}{{\mathfrak Z}}

\newcommand{\bA}{{\mathbb A}}
\newcommand{\bB}{{\mathbb B}}
\newcommand{\bC}{{\mathbb C}}
\newcommand{\bD}{{\mathbb D}}
\newcommand{\bE}{{\mathbb E}}
\newcommand{\bF}{{\mathbb F}}
\newcommand{\bG}{{\mathbb G}}
\newcommand{\bH}{{\mathbb H}}
\newcommand{\bI}{{\mathbb I}}
\newcommand{\bJ}{{\mathbb J}}
\newcommand{\bK}{{\mathbb K}}
\newcommand{\bL}{{\mathbb L}}
\newcommand{\bM}{{\mathbb M}}
\newcommand{\bN}{{\mathbb N}}
\newcommand{\bO}{{\mathbb O}}
\newcommand{\bP}{{\mathbb P}}
\newcommand{\bQ}{{\mathbb Q}}
\newcommand{\bR}{{\mathbb R}}
\newcommand{\bT}{{\mathbb T}}
\newcommand{\bU}{{\mathbb U}}
\newcommand{\bV}{{\mathbb V}}
\newcommand{\bW}{{\mathbb W}}
\newcommand{\bX}{{\mathbb X}}
\newcommand{\bY}{{\mathbb Y}}
\newcommand{\bZ}{{\mathbb Z}}
\newcommand{\id}{{\rm id}}

\title[Colored HOMFLY polynomial via skein theory]{Colored HOMFLY polynomial via skein theory}
\author[Shengmao Zhu]{Shengmao Zhu}
\address{Center of Mathematical Sciences \\Zhejiang University \\Hangzhou, 310027, China }
\email{zhushengmao@gmail.com}

\begin{abstract}
In this paper, we study the properties of the colored HOMFLY
polynomials via HOMFLY skein theory. We prove some limit behaviors
and symmetries of the colored HOMFLY polynomial predicted in some
physicists' recent works.
\end{abstract}

\maketitle



\section{Introduction}
The HOMFLY polynomial is a two variables link invariants which was
first discovered by Freyd-Yetter, Lickorish-Millet, Ocneanu, Hoste
and Przytychi-Traczyk. In V. Jones's seminal paper \cite{Jones}, he
obtained the HOMFLY polynomial by studying the representation of
Heck algebra. Given a oriented link $\mathcal{L}$ in $S^3$, its
HOMFLY polynomial $P_{\mathcal{L}}(q,t)$ satisfies the following
crossing changing formula,
\begin{align}
tP_{\mathcal{L}_+}(q,t)-t^{-1}P_{\mathcal{L}_-}(q,t)=(q-q^{-1})P_{\mathcal{L}_0}(q,t)
\end{align}
Given an initial value $P_{\text{unknot}}(q,t)=1$, one can obtain
the HOMFLY polynomial for a given oriented link recursively through
the above formula (1.1).

According to V. Tureav's work \cite{Turaev}, the HOMFLY polynomial
can be obtained from the quantum invariants associated with the
fundamental representation of the quantum group $U_q(sl_N)$ by
letting $q^N=t$. From this view, it is natural to consider the
quantum invariants associated with arbitrary irreducible
representations of $U_q(sl_N)$. Let $q^N=t$, we call such two
variables invariants as the colored HOMFLY polynomials. See
\cite{LZ} for detail definition of the colored HOMFLY polynomials by
quantum group invariants of $U_q(sl_N)$.

The colored HOMFLY polynomial can also be obtained through the
satellite knot. Given a framed knot $\mathcal{K}$ and a diagram $Q$
in the skein $\mathcal{C}$ of the annulus, the satellite knot
$\mathcal{K}\star Q$ of $\mathcal{K}$ is constructed through drawing
$Q$ on the annular neighborhood of $\mathcal{K}$ determined by the
framing. We refer to this construction as decorating $\mathcal{K}$
with the pattern $Q$.

The skein $\mathcal{C}$ has a natural structure as the commutative
algebra. A subalgebra $\mathcal{C}^+\subset\mathcal{C}$ can be
interpreted as the ring of symmetric functions in infinitely many
variables $x_1,..,x_N,..$ \cite{Morton}. In this context, the Schur
function $s_\lambda(x_1,..,x_N,..)$ corresponding the basis element
$Q_\lambda\in \mathcal{C}^+$ obtained by taking the closure of the
idempotent elements $E_\lambda$ in Heck algebra \cite{L}.

According to Aiston, Lukac et al's work \cite{Ai,L2}, the colored
HOMFLY polynomial of $\mathcal{L}$ with $L$ components labeled by
the corresponding partitions $A^1,..,A^L$, can be identified through
the HOMFLY polynomial of the link $\mathcal{L}$ decorated by
$Q_{A^1},..,Q_{A^L}$. Denote $\vec{A}=(A^1,..,A^L)\in
\mathcal{P}^L$, the colored HOMFLY polynomial of the link
$\mathcal{L}$ can be defined by
\begin{align}
W_{\vec{A}}(\mathcal{L};q,t)=q^{-\sum_{\alpha=1}^Lk_{A^\alpha}w(\mathcal{K}_\alpha)}t^{-\sum_{\alpha=1}^{L}|A^\alpha|w(\mathcal{K}_\alpha)}
\langle\mathcal{L}\star\otimes_{\alpha=1}^LQ_{A^\alpha}\rangle
\end{align}
where $w(\mathcal{K}_\alpha)$ is the writhe number of the
$\alpha$-component $\mathcal{K}_\alpha$ of $\mathcal{L}$, the
bracket
$\langle\mathcal{L}\star\otimes_{\alpha=1}^LQ_{A^\alpha}\rangle$
denotes the framed HOMFLY polynomial of the satellite link
$\mathcal{L}\star\otimes_{\alpha=1}^LQ_{A^\alpha}$. See Section 4
for details.

In physics literature, the colored HOMFLY polynomials was described
as the path integral of the Wilson loops in Chern-Simons quantum
field theory \cite{Witten}. It makes the colored HOMFLY polynomials
stay among the central subjects of the modern mathematics and
physics. H. Itoyama, A. Mironov, A. Morozov, And. Morozov started a
program of systematic study of the colored HOMFLY polynomial with
some physical motivations, see \cite{IMMM1,IMMM2} and the references
in these papers. On one hand, they have obtained some explicit
formulas of colored HOMFLY polynomials for some special links. On
the other hand, they also proposed some conjectural formulas for the
structural properties of the general links. Given a knot
$\mathcal{K}$ and a partition $A\in \mathcal{P}$, they defined the
following special polynomials
\begin{align}
H_{A}^\mathcal{K}(t)=\lim_{q\rightarrow
1}\frac{W_{A}(\mathcal{K};q,t)}{W_{A}(\bigcirc;q,t)}
\end{align}
and its dual
\begin{align}
\Delta_{A}^\mathcal{K}(q)=\lim_{t\rightarrow
1}\frac{W_{A}(\mathcal{K};q,t)}{W_{A}(\bigcirc;q,t)}.
\end{align}
After some concrete calculations, they conjectured
\begin{align}
H_{A}^{\mathcal{K}}(t)=H_{(1)}^{\mathcal{K}}(t)^{|A|}\\
\Delta_{A}^{\mathcal{K}}(q)=\Delta_{(1)}^{\mathcal{K}}(q^{|A|}).
\end{align}
In Section 6, we give a proof of the formula (1.5) which is stated
in the following theorem:
\begin{theorem}
Given $\vec{A}=(A^1,..,A^{L})\in \mathcal{P}^L$ and a link
$\mathcal{L}$ with $L$ components $\mathcal{K}_\alpha,
\alpha=1,..,L$, then we have
\begin{align}
H_{\vec{A}}^{\mathcal{L}}(t)=\prod_{\alpha=1}^{L}H_{(1)}^{\mathcal{K}_\alpha}(t)^{|A^\alpha|}.
\end{align}
\end{theorem}
We mention that this theorem was first proved by K. Liu and P. Peng
\cite{LP1} by the cabling technique. Here, we give two independent
simple proofs via skein theory.

As to the formula (1.6), we show that it is incorrect in general,
and we find a counterexample: considering the partition $A=(22)$,
for the given torus knot $T(2,3)$, a direct calculation shows
\begin{align}
\Delta^{T(2,3)}_{(22)}(q)\neq \Delta^{T(2,3)}_{(1)}(q^4).
\end{align}

However, we still believe that the formula (1.6) holds for any knot
with a given hook partition $A$. In fact, we have proved the
following theorem.
\begin{theorem}
Given a torus knot $T(m,n)$, where $m$ and $n$ are relatively prime.
If $A$ is a hook partition. Then we have
\begin{align}
\Delta_{A}^{T(m,n)}(q)=\Delta_{(1)}^{T(m,n)}(q^{|A|}).
\end{align}
\end{theorem}

In \cite{GS}, S. Gukov and M.
Sto$\check{\text{s}}$i$\acute{\text{c}}$ interpreted the knot
homology as the physical description of the space of the open BPS
states. They found a remarkable "mirror symmetries" by calculating
the colored HOMFLY homology for symmetric and anti-symmetric
representations. Decategorified version of these "mirror symmetries"
for colored HOMFLY homology leads to the symmetry of the colored
HOMFLY polynomial, see formula (5.20) in \cite{GS}. In fact, we have
proved the following symmetric property of the colored HOMFLY
polynomial $W_{\vec{A}}(\mathcal{L};q,t)$ for any link
$\mathcal{L}$.

\begin{theorem}
Given a link $\mathcal{L}$ with $L$ components and a partition
vector $\vec{A}=(A^1,..,A^L)\in \mathcal{P}^L$, we have the symmetry
\begin{align}
W_{\vec{A}}(\mathcal{L};q^{-1},t)=(-1)^{\|\vec{A}\|}W_{\vec{A}^{t}}(\mathcal{L};q,t).
\end{align}
\end{theorem}

Theorem 1.3 was first proved by K. Liu and P. Peng \cite{LP1} which
was used to derive the conjectural structure of the colored HOMFLY
polynomial predicted by Labastida-Mari\~no-Ooguri-Vafa \cite{LMV}.
In Section 7, we give an independent simple proof of this theorem
via skein theory.

Lastly, by the construction of the pattern $Q_\lambda\in
\mathcal{C}$ and the definition of colored HOMFLY polynomial (1.2),
we derive the following symmetry directly.
\begin{theorem}
Given a link $\mathcal{L}$ with $L$ components and a partition
vector $\vec{A}=(A^1,..,A^L)\in \mathcal{P}^L$, we have
\begin{align}
W_{\vec{A}}(\mathcal{L};-q^{-1},t)=(-1)^{\sum_{\alpha=1}^Lk_{A^\alpha}}W_{\vec{A}^{t}}(\mathcal{L};q,t).
\end{align}
\end{theorem}

The rest of this paper is organized as follows. In Section 2, we
gives a brief account of the HOMFLY skein theory and introduce some
basic properties of the HOMFLY polynomial which will be used in
Section 6. In Section 3, we gather some definitions and formulas
related to partitions and symmetric functions, then we introduce the
basic elements in the skein of annulus $\mathcal{C}^+$, such as the
Turaev's basis and symmetric function basis of $\mathcal{C}^+$. In
Section 4, we give the definition of the colored HOMFLY polynomials
as the HOMFLY satellite invariants decorated by the elements in
HOMFLY skein of annulus $\mathcal{C}^+$. In Section 5, we introduce
the skein theory descriptions of the colored HOMFLY polynomials of
torus links as showed by Morton and Manchon \cite{MM} and gives the
explicit formula of colored HOMFLY polynomial for torus links
obtained by X. Lin and H. Zheng \cite{LZ}. In Section 6, we
introduce the definition of the "special polynomials" and prove the
Theorem 1.1 and Theorem 1.2 about the properties of these special
polynomials. In the last Section 7, we study the symmetries of the
colored HOMFLY polynomials and prove the Theorem 1.3 and Theorem
1.4.
\\

{\bf Acknowledgements.} This work was supported by China
Postdoctoral Science Foundation 2011M500986. The author would like
to thank Professor Kefeng Liu for bringing their unpublished paper
\cite{LP3} to his attention. Their work motivates the author to
study the HOMFLY skein theory.

\section{The Skein models}
Given a planar surface $F$, the framed HOMFLY skein $\mathcal{S}(F)$
of $F$ is the $\Lambda$-linear combination of orientated tangles in
$F$, modulo the two local relations as showed in Figure 1, where
$z=q-q^{-1}$,
\begin{figure}
\begin{center}
\includegraphics[width=200 pt]{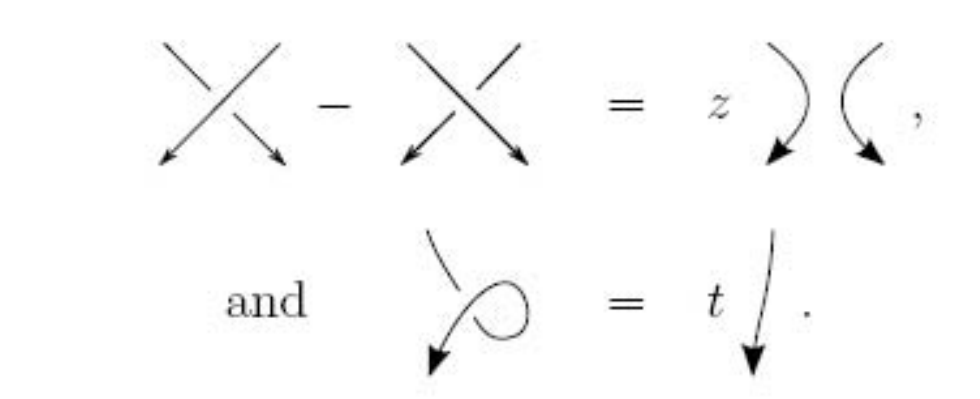}
\caption{}
\end{center}
\end{figure}
the coefficient ring $\Lambda=\mathbb{Z}[q^{\pm 1}, t^{\pm 1} ]$
with the elements $q^{k}-q^{-k}$ admitted as denominators for $k\geq
1$. The local relation showed in Figure 2
\begin{figure}
\begin{center}
\includegraphics[width=160 pt]{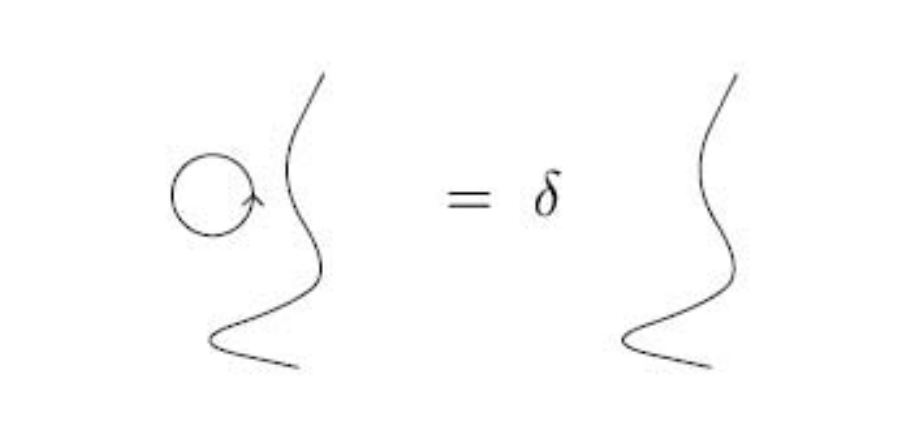}
\caption{}
\end{center}
\end{figure}
is a consequence of the above relations. It follows that the removal
of a null-homotopic closed curve without crossings is equivalent to
time a scalar $\delta=\frac{t-t^{-1}}{q-q^{-1}}$.

\subsection{The plane}
When $F=\mathbb{R}^2$, it is easy to follow that every element in
$\mathcal{S}(F)$ can be represented as a scalar in $\Lambda$. For a
link $\mathcal{L}$ with  diagram $D_{\mathcal{L}}$ the resulting
scalar $\langle D_{\mathcal{L}} \rangle \in \Lambda$ is the framed
HOMFLY polynomial of link $\mathcal{L}$. In the following, we will
also use the notation $\mathcal{L}$ to denote the $D_{\mathcal{L}}$
for simplicity. The unreduced HOMFLY polynomial is obtained by
\begin{align}
\overline{P}_\mathcal{L}(q,t)=t^{-w(\mathcal{L})}\langle \mathcal{L}
\rangle
\end{align}
where $w(\mathcal{L})$ is the writhe of the link $\mathcal{L}$. In
particular,
\begin{align}
\overline{P}_{\bigcirc}(q,t)=\langle\bigcirc\rangle=\frac{t-t^{-1}}{q-q^{-1}}.
\end{align}
The HOMFLY polynomial is defined by
\begin{align}
P_{\mathcal{L}}(q,t)=\frac{t^{-w(\mathcal{L})}\langle
\mathcal{L}\rangle}{\langle\bigcirc\rangle}.
\end{align}
Particularly, $P_{\bigcirc}(q,t)=1$.
\begin{remark}
In some physical literatures, such as \cite{Mar}, the self-writhe
$\bar{w}(\mathcal{L})$ instead of $w(\mathcal{L})$ is used in the
definition of the HOMFLY polynomial (2.1) and (2.2). The
relationship between them is
\begin{align}
w(\mathcal{L})=\bar{w}(\mathcal{L})-2lk(\mathcal{L})
\end{align}
where $lk(\mathcal{L})$ is the total linking number of the link
$\mathcal{L}$.
\end{remark}

A classical result by Lichorish and Millet \cite{LM} showed that for
a given link $\mathcal{L}$ with $L$ components, the lowest power of
$q-q^{-1}$ in the HOMFLY polynomial $P_{\mathcal{L}}(q,t)$ is $1-L$.
In fact, they proved the following theorem.
\begin{theorem}[Lickorish-Millett] Let $\mathcal{L}$ be a link with
$L$ components. Its HOMFLY polynomial has the following expansion
\begin{align}
P_\mathcal{L}(q,t)=\sum_{g\geq
0}p^\mathcal{L}_{2g+1-L}(t)(q-q^{-1})^{2g+1-L}
\end{align}
 which satisfies
\begin{align}
p_{1-L}^{\mathcal{L}}(t)=t^{2lk(\mathcal{L})}(t-t^{-1})^{L-1}\prod_{\alpha=1}^{L}p_0^{\mathcal{K}_\alpha}(t)
\end{align}
where $p_0^{\mathcal{K}_\alpha}(t)$ is the HOMFLY polynomial of the
$\alpha$-th component of the link $\mathcal{L}$ with $q=1$, i.e.
$p_0^{\mathcal{K}_\alpha}(t)=P_{\mathcal{K}_\alpha}(1,t)$.
\end{theorem}
By the definition in our notation (2.3), we have
\begin{align}
\langle\mathcal{L} \rangle=\sum_{g\geq
0}\hat{p}^{\mathcal{L}}_{2g+1-L}(t)(q-q^{-1})^{2g-L}
\end{align}
where
$\hat{p}^{\mathcal{L}}_{2g+1-L}(t)=t^{w(\mathcal{L})}p^{\mathcal{L}}_{2g+1-L}(t)(t-t^{-1})$.
Hence
\begin{align}
\hat{p}^{\mathcal{L}}_{1-L}(t)=t^{\bar{w}(\mathcal{L})}(t-t^{-1})^L\prod_{\alpha=1}^{L}p_{0}^{\mathcal{K}_\alpha}(t)
\end{align}
by the formulas (2.4) and (2.6).

We also need another important property of HOMFLY polynomial.
Denoted by $\mathcal{K}_1\#\mathcal{K}_2$ the connected sum of two
knots $\mathcal{K}_1$ and $\mathcal{K}_2$, then we have
\begin{align}
P_{\mathcal{K}_1
\#\mathcal{K}_2}(q,t)=P_{\mathcal{K}_1}(q,t)P_{\mathcal{K}_2}(q,t).
\end{align}
It is equivalent to say
\begin{align}
\frac{\langle\mathcal{K}_1\#
\mathcal{K}_2\rangle}{\langle\bigcirc\rangle}=\frac{\langle
\mathcal{K}_1\rangle}{\langle\bigcirc\rangle}\frac{\langle\mathcal{K}_2\rangle}{\langle
\bigcirc\rangle}.
\end{align}

\subsection{The rectangle}
When $F$ is a rectangle with $n$ inputs at the top and $n$ outputs
at the bottom. Let $H_n$ be the skein $\mathcal{S}(F)$ of
$n$-tangles. Composing $n$-tangles by placing one above another
induces a product which makes $H_n$ into the Hecke algebra $H_n(z)$
with the coefficients ring $\Lambda$, where $z=q-q^{-1}$. $H_n(z)$
has a presentation generated by the elementary brads $\sigma_i$
subjects to the braid relations
\begin{align}
\sigma_i\sigma_{i+1}\sigma_i=\sigma_{i+1}\sigma_i\sigma_{i+1}\\\nonumber
\sigma_i\sigma_j=\sigma_j\sigma_i, |i-j|\geq 1.
\end{align}
 and the quadratic relations $\sigma_i^2=z\sigma_i+1$.

\subsection{The annulus}
When $F=S^1\times I$ is the annulus, we let
$\mathcal{C}=\mathcal{S}(S^1\times I)$. The skein $\mathcal{C}$ has
a product induced by placing one annulus outside another, under
which $\mathcal{C}$ becomes a commutative algebra. Turaev showed
that $\mathcal{C}$ is freely generated as an algebra by the set
$\{A_m: m\in \mathbb{Z}\}$ where $A_m, m\neq 0$ is represented by
the closure of the braid $\sigma_{|m|-1}\cdots \sigma_2\sigma_1$.
The orientation of the curve around the annulus is counter-clockwise
for positive $m$ and clockwise for negative $m$. The element $A_0$
is the identity element and is represented by the empty diagram.
Thus the algebra $\mathcal{C}$ is the product of two subalgebras
$\mathcal{C}^{+}$ and $\mathcal{C}^{-}$ generated by $\{A_m: m\in
\mathbb{Z}, m\geq 0\}$ and $\{A_m:m\in \mathbb{Z}, m\leq 0\}$.

The closure map $H_n \rightarrow \mathcal{C}^+$, induced by taking
an $n$-tangle $T$ to its closure $\hat{T}$ is a $\Lambda$-linear
map, whose image is denoted by $\mathcal{C}_n$. Thus
$\mathcal{C}^+=\cup_{n\geq 0}\mathcal{C}_n$. There is a good basis
of $\mathcal{C}^+$ consisting of closures of certain idempotents of
$H_n$. In fact, the linear subspace $\mathcal{C}_n$ has a useful
interpretation as the space of symmetric polynomials of degree $n$
in variables $x_1,..,x_N$, for large enough $N$. $\mathcal{C}^+$ can
be viewed as the algebra of the symmetric functions.
\subsection{Involution on the skein of $F$}
The mirror map in the skein of $F$ is defined as the conjugate
linear involution $\bar{\quad}$ on the skein of $F$ induced by
switching all crossings on diagrams and inverting $q$ and $t$ in
$\Lambda$. Thus $\bar{z}=-z$.

\section{Basic elements in the skein of annulus $\mathcal{C}^+$}
\subsection{Partition and symmetric function}
A partition $\lambda$ is a finite sequence of positive integers
$(\lambda_1,\lambda_2,..)$ such that
\begin{align}
\lambda_1\geq \lambda_2\geq\cdots
\end{align}
The length of $\lambda$ is the total number of parts in $\lambda$
and denoted by $l(\lambda)$. The degree of $\lambda$ is defined by
\begin{align}
|\lambda|=\sum_{i=1}^{l(\lambda)}\lambda_i.
\end{align}
If $|\lambda|=d$, we say $\lambda$ is a partition of $d$ and denoted
as $\lambda\vdash d$. The automorphism group of $\lambda$, denoted
by Aut($\lambda$), contains all the permutations that permute parts
of $\lambda$ by keeping it as a partition. Obviously, Aut($\lambda$)
has the order
\begin{align}
|\text{Aut}(\lambda)|=\prod_{i=1}^{l(\lambda)}m_i(\lambda)!
\end{align}
where $m_i(\lambda)$ denotes the number of times that $i$ occurs in
$\lambda$. We can also write a partition $\lambda$ as
\begin{align}
\lambda=(1^{m_1(\lambda)}2^{m_2(\lambda)}\cdots).
\end{align}

Every partition can be identified as a Young diagram. The Young
diagram of $\lambda$ is a graph with $\lambda_i$ boxes on the $i$-th
row for $j=1,2,..,l(\lambda)$, where we have enumerate the rows from
top to bottom and the columns from left to right.

Given a partition $\lambda$, we define the conjugate partition
$\lambda^t$ whose Young diagram is the transposed Young diagram of
$\lambda$ which is derived from the Young diagram of $\lambda$ by
reflection in the main diagonal.

Denote by $\mathcal{P}$ the set of all partitions. We define the
$n$-th Cartesian product of $\mathcal{P}$ as
$\mathcal{P}^n=\mathcal{P}\times \cdots \times\mathcal{P}$. The
elements in $\mathcal{P}^n$ denoted by $\vec{A}=(A^1,..,A^n)$ are
called partition vectors.

The following numbers associated with a given partition $\lambda$
are used frequently in this paper:
\begin{align}
z_\lambda&=\prod_{j=1}^{l(\lambda)}j^{m_{j}(\lambda)}m_j(\lambda)!,\\
k_{\lambda}&=\sum_{j=1}^{l(\lambda)}\lambda_j(\lambda_j-2j+1).
\end{align}
Obviously, $k_\lambda$ is an even number and
$k_\lambda=-k_{\lambda^t}$.

The $m$-th complete symmetric function $h_m$ is defined by its
generating function
\begin{align}
H(t)=\sum_{m\geq 0}h_mt^m=\prod_{i\geq 1}\frac{1}{(1-x_it)}.
\end{align}
The $m$-th elementary symmetric function $e_m$ is defined by its
generating function
\begin{align}
E(t)=\sum_{m\geq 0}e_mt^m=\prod_{i\geq 1}(1+x_it).
\end{align}
Obviously,
\begin{align}
H(t)E(-t)=1.
\end{align}

The power sum symmetric function of infinite variables
$x=(x_1,..,x_N,..)$ is defined by
\begin{align}
p_{n}(x)=\sum_{i}x_i^n.
\end{align}
Given a partition $\lambda$, define
\begin{align}
p_\lambda(x)=\prod_{j=1}^{l(\lambda)}p_{\lambda_j}(x).
\end{align}
The Schur function $s_{\lambda}(x)$ is determined by the Frobenius
formula
\begin{align}
s_\lambda(x)=\sum_{|\mu|=|\lambda|}\frac{\chi_{\lambda}(C_\mu)}{z_\mu}p_\mu(x).
\end{align}
where $\chi_\lambda$ is the character of the irreducible
representation of the symmetric group $S_{|\mu|}$ corresponding to
$\lambda$. $C_\mu$ denotes the conjugate class of symmetric group
$S_{|\mu|}$ corresponding to partition $\mu$. The orthogonality of
character formula gives
\begin{align}
\sum_A\frac{\chi_A(C_\mu) \chi_A(C_\nu)}{z_\mu}=\delta_{\mu \nu}.
\end{align}

We also have the following Giambelli (or Jacobi-Trudi) formula:
\begin{align}
s_\lambda(x)&=\det(h_{\lambda_i-i+j})_{1\leq i,j\leq
l(\lambda)}\\\nonumber &=\det(e_{\lambda^t_i-i+j})_{1\leq i,j\leq
l(\lambda^t)}.
\end{align}

\subsection{Turaev's geometrical basis of $\mathcal{C}^+$}
The element $A_m\in \mathcal{C}^+$ is the closure of the braid
$\sigma_{m-1}\cdots \sigma_2\sigma_1\in H_m$. Its mirror image
$\bar{A}_m$ is the closure of the braid $\sigma_{m-1}^{-1}\cdots
\sigma_2^{-1}\sigma_1^{-1}$. Given a partition
$\lambda=(\lambda_1,..,\lambda_{l})$ of $m$ with length $l$, we
define the monomial $A_\lambda=A_{\lambda_1}\cdots A_{\lambda_{l}}$.
Then the monomials $\{A_\lambda\}_{\lambda \vdash m}$ becomes a
basis of $\mathcal{C}_m$ which is called the Turaev's geometric
basis of $\mathcal{C}^+$.

Moreover, let $A_{i,j}$ be the closure of the braid
$\sigma_{i+j}\sigma_{i+j-1}\cdots \sigma_{j+1}\sigma_{j}^{-1}\cdots
\sigma_1^{-1}$. We define the element $X_m$ in $\mathcal{C}_m$ as
the sum of $m$ closed $m$-braids
\begin{align}
X_m=\sum_{j=0}^{m-1}A_{m-1-j,j}.
\end{align}
There exist some explicit geometric relations between the elements
$\bar{A}_m$, $A_m$ and $X_m$ \cite{MM}.

\subsection{Symmetric function basis of $\mathcal{C}^+$}
The subalgebra $\mathcal{C}^+\subset \mathcal{C}$ can be interpreted
as the ring of symmetric functions in infinite variables
$x_1,..,x_N,..$ \cite{L}. In this subsection, we introduce the
elements in $\mathcal{C}^+$ representing the complete and elementary
symmetric functions $h_m$, $e_m$ and power sum $P_m$.

Given a permutation $\pi\in S_m$ with the length $l(\pi)$, let
$\omega_\pi$ be the positive permutation braid associated to $\pi$.
We define two basis quasi-idempotent elements in $H_m$:
\begin{align}
a_m=\sum_{\pi\in S_m}q^{l(\pi)}\omega_{\pi}\\
b_m=\sum_{\pi\in S_m}(-q)^{-l(\pi)}\omega_{\pi}
\end{align}

The element $h_m\in \mathcal{C}_m$ which represents the complete
symmetric function with degree $m$, is the closure of the elements
$\frac{1}{\alpha_m}a_m\in H_m$. Where $\alpha_m$ is determined by
the equation $a_ma_m=\alpha_m a_m$, it gives
$\alpha_m=q^{m(m-1)/2}\prod_{i=1}^m\frac{q^i-q^{-i}}{q-q^{-1}}$.
Similarly, the closure of the element $\frac{b_m}{\beta_m}$ gives
the elements $e_m\in \mathcal{C}_m$ represents the elementary
symmetric function, where $\beta_m=\alpha_m|_{q\rightarrow
-q^{-1}}$. $\{h_n\}$ generates the skein module $\mathcal{C}^+$, and
the monomial $h_{\lambda}$, where $|\lambda|=m$, form a basis for
$\mathcal{C}_m$. Then $\mathcal{C}^+$ can be regarded as the ring of
symmetric functions in variables $x_1,..,x_N,..$ with the
coefficient ring $\Lambda$. In this situation, $\mathcal{C}_m$
consists of the homogeneous functions of degree $m$.

The power sum $P_m=\sum x_i^m$ are symmetric functions which can be
represented in terms of the complete symmetric functions, hence
$P_m\in \mathcal{C}_m$. Moreover, we have the identity
\begin{align}
\{m\}P_m=X_m.
\end{align}
where $\{m\}=\frac{q^m-q^{-m}}{q-q^{-1}}$. Denoted by $Q_{\lambda}$
the closures of Aiston's idempotent elements $e_{\lambda}$ in the
Hecke algebra $H_m$. It was showed by Lukac \cite{L} that
$Q_\lambda$ represent the Schur functions in the interpretation as
symmetric functions. Hence
\begin{align}
Q_\lambda=\det(h_{\lambda_i+j-i})_{1\leq i,j \leq l}
\end{align}
where $\lambda=(\lambda_1,..,\lambda_l)$. In particularly, we have
$Q_{\lambda}=h_m$, when $\lambda=(m)$ is a row partition, and
$Q_{\lambda}=e_m$ when $\lambda=(1,...,1)$ is a column partition.
$\{Q_{\lambda}\}_{\lambda\vdash m}$ forms a basis of
$\mathcal{C}_m$. Furthermore, the Frobenius formula (3.12) gives:
\begin{align}
Q_\lambda=\sum_{\mu}\frac{\chi_{\lambda}(C_\mu)}{z_{\mu}}P_{\mu}
\end{align}
where
\begin{align}
P_{\mu}=\prod_{i=1}^{l(\mu)}P_{\mu_i}.
\end{align}

\section{Colored HOMFLY polynomials}
Let $\mathcal{L}$ be a framed link with $L$ components with a fixed
numbering. For diagrams $Q_1,..,Q_L$ in the skein model of annulus
with the positive oriented core $\mathcal{C}^+$, we define the
decoration of $\mathcal{L}$ with $Q_1,..,Q_L$ as the link
\begin{align}
\mathcal{L}\star \otimes_{i=1}^{L} Q_i
\end{align}
which derived from $\mathcal{L}$ by replacing every annulus
$\mathcal{L}$ by the annulus with the diagram $Q_i$ such that the
orientations of the cores match. Each $Q_i$ has a small backboard
neighborhood in the annulus which makes the decorated link
$\mathcal{L}\otimes_{i=1}^{L}Q_i$ into a framed link.

The framed colored HOMFLY polynomial of $\mathcal{L}$ is defined to
be the framed HOMFLY polynomial of the decorated link
$\mathcal{L}\star\otimes_{i=1}^{L}Q_i$, i.e.
\begin{align}
\langle \mathcal{L}\star\otimes_{i=1}^{L}Q_i \rangle.
\end{align}

In particular, when $Q_{A^\alpha}\in \mathcal{C}_{d_\alpha}$, where
$A^\alpha$ is the partition of a positive integer $d_\alpha$, for
$\alpha=1,..,L$. We add a framing factor to eliminate the framing
dependency. It makes the framed colored HOMFLY polynomial $ \langle
\mathcal{L}\otimes_{\alpha=1}^{L}Q_{A^\alpha}\rangle $ into a
framing independent invariant which is given by
\begin{align}
W_{\vec{A}}(\mathcal{L};q,t)=q^{-\sum_{\alpha=1}^{L}\kappa_{A^\alpha}|w(\mathcal{K_\alpha})|}
t^{-\sum_{\alpha=1}^{L}|A^\alpha|w(\mathcal{K}_\alpha)}\langle
\mathcal{L}\star\otimes_{\alpha=1}^{L}Q_{A^{\alpha}} \rangle
\end{align}
where $\vec{A}=(A^1,..,A^L)\in \mathcal{P}^L$.

From now on, we will call $W_{\vec{A}}(\mathcal{L};q,t)$ the colored
HOMFLY polynomial of link $\mathcal{L}$ with the color
$\vec{A}=(A^1,..,A^L)$.

\begin{example}
The following examples are some special cases of the colored HOMFLY
polynomials of links.

(1). The unknot $\bigcirc$,
\begin{align}
W_A(\bigcirc;q,t)=\langle
Q_A\rangle=s_A^*(q,t)=\sum_{B}\frac{\chi_{A}(C_B)}{z_B}\prod_{j=1}^{l(B)}\frac{t^{B_j}-t^{-B_j}}{q^{B_{j}}-q^{-B_j}}
\end{align}

(2). When $A^1=A^2=\cdots=A^L=(1)$,
\begin{align}
W_{((1),..,(1))}(\mathcal{L};q,t)&=t^{-\sum_\alpha
w(\mathcal{K}_\alpha)}\langle\mathcal{L}\rangle\\\nonumber
&=t^{-2lk(\mathcal{L})}\langle\bigcirc\rangle
t^{-w(\mathcal{L})}\frac{\langle
\mathcal{L}\rangle}{\langle\bigcirc\rangle}\\\nonumber
&=t^{-2lk(\mathcal{L})}\left(\frac{t-t^{-1}}{q-q^{-1}}\right)P_\mathcal{L}(q,t).
\end{align}

(3). When $\mathcal{L}$ is the disjoint union of $L$ knots, i.e.
$\mathcal{L}=\otimes_{\alpha=1}^L \mathcal{K}_\alpha$,
\begin{align}
W_{(A^1,..,A^L)}(\mathcal{L};q,t)&=q^{-\sum
k_{A^{\alpha}}w(\mathcal{K}_\alpha)}t^{-\sum
|A^\alpha|w(\mathcal{K}_\alpha)}\langle\otimes_\alpha
\mathcal{K}_\alpha \star Q_{A^\alpha}\rangle\\\nonumber
&=\prod_{\alpha}q^{-k_{A^{\alpha}}w(\mathcal{K}_\alpha)}t^{-|A^\alpha|w(\mathcal{K}_\alpha)}\langle
\mathcal{K}_\alpha\star Q_{A^\alpha}\rangle\\\nonumber
&=\prod_\alpha W_{A^\alpha}(\mathcal{K}_\alpha;q,t).
\end{align}
\end{example}

\section{Decorated torus links}
Given $T\in \mathcal{C}_m$ which is the closure of a $m$-braid
$\beta$ such that $T$ has $L$ components and with the $\alpha$-th
component consists of $m_{\alpha}$ strings. Thus
$\sum_{\alpha=1}^{L}m_\alpha=m$. Let $Q_{A^{\alpha}}\in
\mathcal{C}_{d_\alpha}$ where $A^{\alpha}\vdash d_\alpha$. It is
clear that $T\star \otimes_{\alpha=1}^{L} Q_{A^\alpha} \in
\mathcal{C}_{n}$, where $n=\sum_{\alpha=1}^{L}m_{\alpha}d_{\alpha}$.
Since $\{Q_\mu\}_{\mu\vdash n}$ forms a linear basis of
$\mathcal{C}_n$ as showed in the last section. We obtain the
following expansion
\begin{align}
T\star \otimes_{\alpha=1}^{L} Q_{A^\alpha}=\sum_{\mu \vdash
n}c_{\lambda}^{\mu}Q_{\mu}
\end{align}

Let us consider the cable link diagram $T=T_{mL}^{nL}$ which is the
closure of the framed $mL$-braid $(\beta_{mL})^{nL}$,  where
$(m,n)=1$. The braid $\beta_{m}$ is showed in Figure 3.

\begin{figure}[b]
\begin{center}
\includegraphics[width=160 pt]{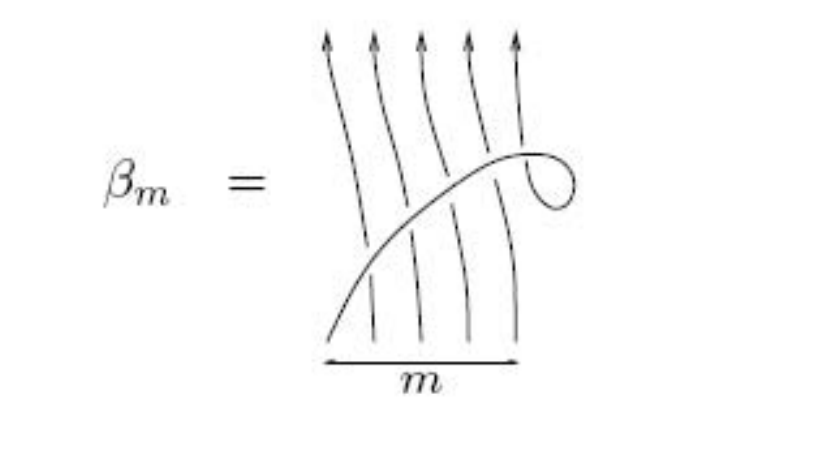} \caption{}
\end{center}
\end{figure}

 $T=T_{mL}^{nL}$ induces a map $F_{mL}^{nL}:
\otimes_{\alpha=1}^{L}\mathcal{C}_{d_\alpha}\rightarrow
\mathcal{C}_{m(\sum_{\alpha=1}^{L}d_{\alpha})}$ by taking an element
$\otimes_{\alpha=1}^L Q_{A^{\alpha}}$ to $T_{mL}^{nL}\star
\otimes_{\alpha=1}^{L} Q_{A^{\alpha}}$.

We define $\tau=F^{1}_1$, then $\tau$ is the framing change map. It
was showed in \cite{MM}  that
\begin{align}
\tau(Q_\mu)=\tau_\mu Q_\mu
\end{align}
where $\tau_\mu=q^{k_\mu}t^{|\mu|}$. The fractional twist map
$\tau^{\frac{n}{m}}:\mathcal{C}^+\rightarrow \mathcal{C}^+$ is the
linear map defined on the basis $Q_{\mu}$ by
\begin{align}
\tau^{\frac{n}{m}}(Q_{\mu})=(\tau_\mu)^{\frac{n}{m}}Q_\mu.
\end{align}
In order to give an expression for
$F^{nL}_{mL}(\otimes_{\alpha=1}^{L}Q_{A^{\alpha}})$, we need to
introduce the terminology of $plethysms$. Given a symmetric
polynomial in $N$ variables $p(x_1,..,x_N)=M_1+\cdots+M_r$ with $r$
monomials $M_i$. Let $q(x_1,..,x_r)$ be a symmetric function in $r$
variables. The plethysm $q[p]$ is the symmetric function of $N$
variables $q[p]=q(M_1,..,M_r)$. Since $\mathcal{C}^+$ is isomorphic
to the ring of symmetric functions. Let $Q\in \mathcal{C}^+$ and let
$P\in \mathcal{C}^+$ represent a sum of monomials each with
coefficient $1$. We use the notation $Q[P]$ to express the element
in $\mathcal{C}^+$ corresponding to the plethym of the functions
represented by $Q$ and $P$.

We have the following formula which is the link version of the
Theorem 13 as showed in \cite{MM}
\begin{align}
F_{mL}^{nL}(\otimes_{\alpha=1}^{L}Q_{A^{\alpha}})=\tau^{\frac{n}{m}}(\prod_{\alpha=1}^{L}Q_{A^{\alpha}}[P_m]).
\end{align}
With the definition of plethsm, we have
\begin{align}
\prod_{\alpha=1}^{L}Q_{A^{\alpha}}[P_m]=\sum_{\mu\vdash
m\sum_{\alpha=1}^Ld_\alpha}C_{A^1,..,A^{L}}^{\mu}Q_{\mu}
\end{align}
where $C_{A^1,..,A^{L}}^\mu$ are the coefficients given by
\begin{align}
s_{A^1}(x_1^m,x_2^m..,)\cdots
s_{A^L}(x_1^m,x_2^m,..)=\sum_{\mu\vdash
m\sum_{\alpha=1}^Ld_\alpha}C_{A^1,..,A^L}^{\mu}s_\mu(x_1,x_2,..).
\end{align}

By the definition of fractional twist map of $\tau^{\frac{n}{m}}$,
we obtain
\begin{align}
F_{mL}^{nL}(\otimes_{\alpha=1}^{L}Q_{A^{\alpha}})=\sum_{\mu\vdash
m\sum_{\alpha=1}^Ld_\alpha}C^\mu_{A^1,..,A^L}q^{\frac{n}{m}k_{\mu}}t^{n\sum_{\alpha}|A^\alpha|}Q_{\mu}.
\end{align}
Therefore, by definition (4.3), the colored HOMFLY polynomial of the
torus link $T_{mL}^{nL}$ is given by
\begin{align}
&W_{A^1,..,A^{L}}(T_{mL}^{nL};q,t)\\
\nonumber &=q^{-m\cdot n\sum_{\alpha=1}^L k_{A^{\alpha}}}t^{-n\cdot
m\sum_{\alpha=1}^L|A^\alpha|}\langle
F_{mL}^{nL}(\otimes_{\alpha=1}^{L}Q_{A^{\alpha}})\rangle \\\nonumber
&=q^{-m\cdot n\sum_{\alpha=1}^L
k_{A^{\alpha}}}t^{-n(m-1)\sum_{\alpha=1}^L|A^\alpha|}\sum_{\mu\vdash
m\sum_{\alpha=1}^L|A^\alpha|}C^\mu_{A^1,..,A^L}q^{\frac{n}{m}k_{\mu}}W_{\mu}(\bigcirc;q,t)
\end{align}
where $W_{\mu}(\bigcirc;q,t)=\langle
Q_\mu\rangle=s^*_{\mu}(q,t)=\sum_{\nu}\frac{\chi_\mu(C_\nu)}{z_\nu}\prod_{i=1}^{l(\nu)}\left(\frac{t^{\nu_i}-t^{-\nu_i}}{q^{\nu_i}-q{-\nu_i}}\right)$
 is the colored HOMFLY polynomial of unknot $\bigcirc$. The coefficients $C_{A^1,..,A^L}^{\mu}$ can be calculated as
follows. According to the Frobenuius formula (3.12), we have
\begin{align}
\prod_{\alpha=1}^L
s_{A^\alpha}(x_1^m,x_2^m,..)&=\sum_{B^1,..,B^L}\prod_{\alpha=1}^L\frac{\chi_{A_\alpha}(C_{B^\alpha})}{z_{B^\alpha}}\prod_{\alpha=1}^L
P_{mB^\alpha}(x_1,x_2..)\\\nonumber
&=\sum_{B^1,..,B^L}\prod_{\alpha=1}^L\frac{\chi_{A_\alpha}(C_{B^\alpha})}{z_{B^\alpha}}
P_{m\sum_{i=1}^L B^i}(x_1,x_2..)\\\nonumber
&=\sum_{B^1,..,B^L}\prod_{\alpha=1}^L\frac{\chi_{A_\alpha}(C_{B^\alpha})}{z_{B^\alpha}}
\sum_{\mu\vdash
m\sum_{\alpha}|B^{\alpha}|}\chi_\mu(C_{m\sum_{i=1}^LB^i})s_{\mu}(x_1,x_2,..).
\end{align}
It follows that
\begin{align}
C_{A^1,..,A^L}^{\mu}=\sum_{B^1,..,B^L}\prod_{\alpha=1}^{L}\frac{\chi_{A^\alpha}(C_{B^\alpha})}{z_{B^\alpha}}\chi_{\mu}(C_{m\sum_{i=1}^LB^i}).
\end{align}
\begin{example}
Substituting  $L=1$, and the partition $A=(1)$ in formula (5.7), we
obtain the following relation in HOMFLY skein $\mathcal{C}_m$:
\begin{align}
T_m^n =T_{m}^n\star Q_{(1)}
=\sum_{\mu}\chi_{\mu}(C_{(m)})q^{\frac{n}{m}k_\mu} t^n  Q_\mu
\end{align}
So one has
\begin{align}
W_{(1)}(T^n_m;q,t)=t^{-n(m-1)}\sum_{\mu}\chi_{\mu}(C_{(m)})q^{\frac{n}{m}k_\mu}W_{\mu}(\bigcirc;q,t)
\end{align}
which is the formula (6.71) showed in \cite{DSV}.
\end{example}
\begin{example}
Torus knot $T(2,2k+1)$,
\begin{align}
W_{(1)}(q,t)&=t^{-(2k+1)}(q^{2k+1}s^*_{(2)}(q,t)-q^{-(2k+1)}s^*_{(11)}(q,t)),\\
W_{(2)}(q,t)&=t^{-4k-2}(q^{4k+2}s^*_{(4)}(q,t)-q^{-4k-2}s^*_{(31)}(q,t)+q^{-8k-4}s^*_{(22)}(q,t)),\\
W_{(11)}(q,t)&=t^{-4k-2}(q^{4k+2}s^*_{(22)}(q,t)-q^{4k+2}s^*_{(211)}(q,t)+q^{-4k-2}s^*_{(1111)}(q,t)).
\end{align}
\end{example}
\begin{example}
Torus link $T(2,2k)$,
\begin{align}
W_{(1),(1)}(q,t)&=q^{2k}s^*_{(2)}(q,t)+q^{-2k}s^*_{(11)}(q,t),\\
W_{(2),(1)}(q,t)&=q^{4k}s^*_{(3)}(q,t)+q^{-2k}s^*_{(21)}(q,t).
\end{align}
\end{example}

\section{Special polynomials}
Given a knot $\mathcal{K}$ and a partition $A\in \mathcal{P}$,  P.
Dunin-Barkowski, A. Mironov, A. Morozov, A. Sleptsov and A. Smirnov
\cite{BMMSS} defined the following special polynomial
\begin{align}
H_{A}^\mathcal{K}(t)=\lim_{q\rightarrow
1}\frac{W_{A}(\mathcal{K};q,t)}{W_{A}(\bigcirc;q,t)}
\end{align}
and its dual
\begin{align}
\Delta_{A}^\mathcal{K}(q)=\lim_{t\rightarrow
1}\frac{W_{A}(\mathcal{K};q,t)}{W_{A}(\bigcirc;q,t)}.
\end{align}
In particular, when $A=(1)$, we have
\begin{align}
H_{(1)}^\mathcal{K}(t)=\lim_{q\rightarrow
1}\frac{W_{(1)}(\mathcal{K};q,t)}{W_{(1)}(\bigcirc;q,t)}=P_{\mathcal{K}}(1,t)
\end{align}
where $P_{\mathcal{K}}(q,t)$ is the HOMFLY polynomial as defined in
(2.3). And
\begin{align}
\Delta_{(1)}^{\mathcal{K}}(q)=\lim_{t\rightarrow
1}\frac{W_{(1)}(\mathcal{K};q,t)}{W_{(1)}(\bigcirc;q,t)}=\Delta_{\mathcal{K}}(q)
\end{align}
where $\Delta_{\mathcal{K}}(q)$ is the Alexander polynomial of the
knot $\mathcal{K}$.

After testing many examples \cite{BMMSS,IMMM1,IMMM2}, they proposed
the following conjecture.
\begin{conjecture}
Given a knot $\mathcal{K}$ and partition $A\in \mathcal{P}$, we have
the following identities:
\begin{align}
H_{A}^{\mathcal{K}}(t)=H_{(1)}^{\mathcal{K}}(t)^{|A|}\\
\Delta_{A}^{\mathcal{K}}(q)=\Delta_{(1)}^{\mathcal{K}}(q^{|A|}).
\end{align}
\end{conjecture}

In the followings, we give a proof the formula (6.5). In fact, one
can generalize the definition of the special polynomial (6.1) for
any link $\mathcal{L}$ with $L$ components
\begin{align}
H_{\vec{A}}^\mathcal{L}(t)=\lim_{q\rightarrow
1}\frac{W_{\vec{A}}(\mathcal{L};q,t)}{W_{\vec{A}}(\bigcirc^{\otimes
L};q,t)}.
\end{align}
In fact, the formula (6.5) is the special case of the following
theorem.
\begin{theorem}
Given $\vec{A}=(A^1,..,A^{L})\in \mathcal{P}^L$ and a link
$\mathcal{L}$ with $L$ components $\mathcal{K}_\alpha,
\alpha=1,..,L$, then we have
\begin{align}
H_{\vec{A}}^{\mathcal{L}}(t)=\prod_{\alpha=1}^{L}H_{(1)}^{\mathcal{K}_\alpha}(t)^{|A^\alpha|}.
\end{align}
\end{theorem}
We will give two proofs of Theorem 6.2. The first proof:
\begin{proof}
Choosing a crossing $c$ of the link $\mathcal{L}$, suppose $c$ is a
positive crossing, by the skein relation,
\begin{align}
\mathcal{L}(c)=\mathcal{L}(c^{-1})+(q-q^{-1})\mathcal{L}(\upuparrows)
\end{align}
It is clear that $\mathcal{L}(c)$ and $\mathcal{L}(c^{-1})$ have the
same number of link components.

Considering the degree of $q-q^{-1}$ in the expansion form (2.5) for
the HOMFLY polynomial, if the number of the link components for
$\mathcal{L}(\upuparrows)$ is not big than $\mathcal{L}(c)$, then
taking the limit $q\rightarrow 1$ after the removal of singularity,
we have
\begin{align}
\mathcal{L}(c)=\mathcal{L}(c^{-1}).
\end{align}

Therefore, with the above rule, one can exchange the crossings
between different components of link $\mathcal{L}$, such that
$\mathcal{L}=\otimes_{\alpha=1}^{L} \mathcal{K}_\alpha$. In this
case, as showed in the formula (4.6),
\begin{align}
W_{\vec{A}}(\mathcal{L};q,t)=\prod_{\alpha=1}^LW_{A^{\alpha}}(\mathcal{K}_\alpha;q,t).
\end{align}
Similarly, $ W_{\vec{A}}(\bigcirc^{\otimes
L};q,t)=\prod_{\alpha=1}^{L}W_{A^{\alpha}}(\bigcirc;q,t)$, we obtain
\begin{align}
\lim_{q\rightarrow
1}\frac{W_{\vec{A}}(\mathcal{L};q,t)}{W_{\vec{A}}(\bigcirc^{\otimes
L};q,t)}=\prod_{\alpha=1}^{L}\lim_{q\rightarrow
1}\frac{W_{A^\alpha}(\mathcal{K}_\alpha;q,t)}{W_{A^\alpha}(\bigcirc;q,t)}.
\end{align}
Thus, we only need to consider the case of each knot
$\mathcal{K}_\alpha$.

One can exchange the crossings which lie in $\mathcal{K}_\alpha\star
Q_{A^\alpha}$ but not in $Q_{A^\alpha}$ by the skein relation (6.9).
These crossings satisfy the relation (6.10) in the limit
$q\rightarrow 1$. Hence
\begin{align}
\mathcal{K}_\alpha\star
Q_{A^\alpha}=\mathcal{K}_{\alpha}\#\mathcal{K}_\alpha\# \cdots \#
\mathcal{K}_\alpha \# Q_{A^\alpha}
\end{align}
where the righthand side of (6.13) represents the connected sum of
$Q_{A^\alpha}$ and $|A^\alpha|$ knots $\mathcal{K}_\alpha$. By the
formula (2.10), we have
\begin{align}
\frac{\langle{\mathcal{K}_\alpha \star Q_{A^\alpha}}\rangle}{\langle
\bigcirc\rangle}=\frac{\langle
Q_{A^\alpha}\rangle}{\langle\bigcirc\rangle}\left(\frac{\langle\mathcal{K}_\alpha\rangle}{\langle\bigcirc\rangle}\right)^{
|A^\alpha|}.
\end{align}
Therefore,  we obtain
\begin{align}
\lim_{q\rightarrow
1}\frac{W_{A^\alpha}(\mathcal{K}_\alpha;q,t)}{W_{A^\alpha}(\bigcirc;q,t)}&=\lim_{q\rightarrow
1}\frac{t^{-|A^\alpha|w(\mathcal{K}_\alpha)}\langle
\mathcal{K}_\alpha \star Q_{A^\alpha}\rangle}{\langle
Q_{A^\alpha}\rangle}\\\nonumber & =\lim_{q\rightarrow
1}\left(\frac{t^{-w(\mathcal{K}_\alpha)}\langle
\mathcal{K}_\alpha\rangle}{\langle\bigcirc\rangle}\right)^{|A^\alpha|}\\\nonumber
&=P_{\mathcal{K}_\alpha}(1,t)^{|A^\alpha|}
\end{align}
which is just the formula (6.7).
\end{proof}

The second proof.
\begin{proof}
We only give the proof for the case of a knot $\mathcal{K}$. It is
easy to generalize the proof for any link $\mathcal{L}$. Given a
partition $A$ with $|A|=d$, by definition
\begin{align}
W_A(\mathcal{K};q,t)&=q^{-k_Aw(\mathcal{K})}t^{-dw(\mathcal{K})}\langle\mathcal{K}\star
Q_A\rangle\\\nonumber
&=q^{-k_Aw(\mathcal{K})}t^{-dw(\mathcal{K})}\sum_{B}\frac{\chi_{A}(B)}{z_B}\langle\mathcal{K}\star
P_B\rangle\\\nonumber
&=q^{-k_Aw(\mathcal{K})}t^{-dw(\mathcal{K})}\left(\frac{\chi_{A}(C_{(1^d)})}{z_{(1^d)}}\langle\mathcal{K}\star
P_{(1^d)}\rangle+\sum_{l(B)\leq
d-1}\frac{\chi_{A}(C_B)}{z_{B}}\langle\mathcal{K}\star
P_B\rangle\right)
\end{align}
and
\begin{align}
s_A^{*}(q,t)&=\sum_{B}\frac{\chi_{A}(B)}{z_B}\prod_{j=1}^{l(B)}\frac{t^{B_j}-t^{-B_j}}{q^{B_j}-q^{-B_j}}\\\nonumber
&=\left(\frac{\chi_{A}(C_{(1^d)})}{z_{(1^d)}}\left(\frac{t-t^{-1}}{q-q^{-1}}\right)^{d}+\sum_{l(B)\leq
d-1}\frac{\chi_{A}(C_B)}{z_{B}}\prod_{j=1}^{l(B)}\frac{t^{B_j}-t^{-B_j}}{q^{B_j}-q^{-B_j}}\right)
\end{align}

By the formula (3.21), it is clear that
\begin{align}
\mathcal{K}\star P_{(1^d)}=\mathcal{K}\star P_{(1)}\cdots P_{(1)}
\end{align}
and the number of the link components is $L(\mathcal{K}\star
P_{(1)}\cdots P_{(1)})=d$. For $|B|=d$, with $l(B)\leq d-1$, the
number of the link components $L(\mathcal{K}\star P_B)$ is less than
$d-1$.

According to the expansion formula (2.7), we have
\begin{align}
\langle\mathcal{K}\star P_{(1^d)}\rangle=\sum_{g\geq
0}\hat{p}_{2g+1-d}^{\mathcal{K}\star P_{(1^d)}}(t)(q-q^{-1})^{2g-d}
\end{align}
and for $l(B)\leq d-1$,
\begin{align}
\langle\mathcal{K}\star P_{B}\rangle=\sum_{g\geq
0}\hat{p}_{2g+1-L(\mathcal{K}\star P_{B})}^{\mathcal{K}\star
P_{B}}(t)(q-q^{-1})^{2g-L(\mathcal{K}\star P_{B})}
\end{align}
with link components $L(\mathcal{K}\star P_{B})\leq d-1$.

Since $\chi_{A}(C_{(1^d)})\neq 0$, by a direct calculation, we
obtain
\begin{align}
\lim_{q\rightarrow
1}\frac{W_A(\mathcal{K};q,t)}{s_A^{*}(q,t)}=\frac{t^{-dw(\mathcal{K})}\hat{p}_0^{\mathcal{K}\star
P_{(1^d)}}(t)}{(t-t^{-1})^d}
\end{align}

According to the formula (2.8)
\begin{align}
\hat{p}_0^{\mathcal{K}\star
P_{(1^d)}}(t)&=t^{\bar{w}(\mathcal{K}\star
P_{(1^d)})}(t-t^{-1})^{d}(p_0^{\mathcal{K}}(t))^{d}.
\end{align}
Moreover, it is clear that $\bar{w}(\mathcal{K}\star
P_{(1^d)})=dw(\mathcal{K})$, thus
\begin{align}
\lim_{q\rightarrow
1}\frac{W_A(\mathcal{K};q,t)}{s_A^{*}(q,t)}=p_0^{\mathcal{K}}(t)^d=P_\mathcal{K}(1,t)^d.
\end{align}
\end{proof}

\begin{example}
For the torus knot $T(2,2k+1)$. Its HOMFLY polynomial is
\begin{align}
P_{T(2,2k+1)}(q,t)=\left(\frac{q^{2k+2}-q^{-2k-2}}{q^2-q^{-2}}t^{-2k}-\frac{q^{2k}-q^{-2k}}{q^2-q^{-2}}t^{-2k-2}\right).
\end{align}
Hence
\begin{align}
P_{T(2,2k+1)}(1,t)=(k+1)t^{-2k}-kt^{-2k-2}.
\end{align}
By formulas (5.13), (5.14) and (5.15), we obtain
\begin{align}
\lim_{q\rightarrow
1}\frac{W_{(1)}(q,t)}{s^*_{(1)}(q,t)}&=\frac{1}{2}t^{-2-2k}(1+t^2+(2k+1)(t^2-1))=P_{T(2,2k+1)}(1,t)\\
\lim_{q\rightarrow
1}\frac{W_{(2)}(q,t)}{s^*_{(2)}(q,t)}&=\frac{1}{4}t^{-4-4k}(1+t^2+(2k+1)(t^2-1))^2=P_{T(2,2k+1)}(1,t)^2\\
\lim_{q\rightarrow
1}\frac{W_{(11)}(q,t)}{s^*_{(11)}(q,t)}&=\frac{1}{4}t^{-4-4k}(1+t^2+(2k+1)(t^2-1))^2=P_{T(2,2k+1)}(1,t)^2
\end{align}
\end{example}

As to the formula (6.6) in Conjecture 6.1, we found that this
formula does not hold for arbitrary partition $A$. Given the torus
knot $T(2,3)$, its Alexander polynomial is
\begin{align}
\Delta^{T(2,3)}_{(1)}(q)=\Delta_2-1.
\end{align}
where we have used the notation $\Delta_d=q^d-q^{-d}$. Considering
the partition $A=(22)$, we have
\begin{align}
\Delta^{T(2,3)}_{(22)}(q)&=\lim_{t\rightarrow
1}\frac{W_{(22)}(q,t)}{s^*_{(22)}(q,t)}\\\nonumber &=
8-7\Delta_4-\Delta_6+6\Delta_8+2\Delta_{10}-5\Delta_{12}\\\nonumber
&
-2\Delta_{14}+3\Delta_{16}+\Delta_{18}-\Delta_{20}-\Delta_{22}\\\nonumber
&\neq \Delta^{T(2,3)}_{(1)}(q^4).
\end{align}
However, we believe that the formula (6.6) holds for any knots when
$A$ is a hook partition. In fact, we have proved the following
theorem.
\begin{theorem}
Given a torus knot $T(m,n)$, where $m$ and $n$ are relatively prime.
If $A$ is a hook partition, then we have
\begin{align}
\Delta_{A}^{T(m,n)}(q)=\Delta_{(1)}^{T(m,n)}(q^{|A|}).
\end{align}
\end{theorem}
Every hook partition can be presented as the form $(a+1,1,...,1)$
with $b+1$ length for some $a,b\in \mathbb{Z}_{\geq 0}$, denoted by
$(a|b)$.

Before to prove this theorem, we need to introduce the following
lemma first.
\begin{lemma}
Given a partition $B$, we have the following identity,
\begin{align}
\sum_{a+b+1=|B|}\chi_{(a|b)}(C_B)(-1)^bu^{a-b}=\frac{\prod_{j=1}^{l(B)}(u^{B_j}-u^{-B_j})}{u-u^{-1}}.
\end{align}
\end{lemma}
\begin{proof}
According to problem 14 at page 49 of \cite{Mac}, taking $t=u$, we
have
\begin{align}
\prod_{i}\frac{1-u^{-1}x_i}{1-ux_i}=E(-u^{-1})H(u)=1+(u-u^{-1})s_{(a|b)}(x)(-1)^bu^{a-b}
\end{align}
since
$s_{(a|b)}(x)=\sum_{\lambda}\frac{\chi_{(a|b)}(C_\lambda)}{z_\lambda}p_\lambda(x)$,
and
\begin{align}
E(-u^{-1})H(u)&=\frac{H(u)}{H(u^{-1})}\\\nonumber
&=\exp\left(\sum_{r\geq
1}\frac{p_r(x)}{r}(u^r-u^{-r})\right)\\\nonumber &=\prod_{r\geq
1}\exp\left(\frac{p_r(x)}{r}(u^r-u^{-r})\right)\\\nonumber
&=\prod_{r\geq 1}\sum_{m_r\geq
0}\frac{p_r(x)^{m_r}(u^r-u^{-r})^{m_r}}{r^{m_r}m_r!}\\\nonumber
&=\sum_{\lambda}\frac{p_{\lambda}(x)}{z_\lambda}\prod_{j=1}^{l(\lambda)}(u^{\lambda_j}-u^{-\lambda_j})
\end{align}
Comparing the coefficients of $p_B(x)$ in (6.34), the formula (6.32)
is obtained.
\end{proof}
In the following, we assume $A$ and $B$ are two partitions of $d$
and $\mu$ is a partition of $md$. By the property of character
theory of symmetric group, one has

\makeatletter
\let\@@@alph\@alph
\def\@alph#1{\ifcase#1\or \or $'$\or $''$\fi}\makeatother
\begin{subnumcases}
{\chi_A(C_{(d)})=} (-1)^b, &\text{if $A$ is a hook partition
$(a|b)$}\\\nonumber 0, &\text{otherwise}
\end{subnumcases}
\makeatletter\let\@alph\@@@alph\makeatother

The following formula is a consequence of Lemma 6.5,
\begin{align}
\sum_{\mu}\chi_{\mu}(C_{mB})\chi_{\mu}(C_{(md)})q^{\frac{n}{m}k_\mu}&=\sum_{a'+b'+1=md}\chi_{(a'|b')}(C_{mB})(-1)^{b'}q^{\frac{n}{m}(a'-b')md}\\\nonumber
&=\frac{\prod_{j=1}^{l(B)}(q^{mndB_j}-q^{-mndB_j})}{q^{nd}-q^{-nd}}.
\end{align}

Given a hook partition $A=(a|b)$, according to Lemma 6.5, we also
have
\begin{align}
&\sum_{B}\frac{\chi_{(a|b)}(C_B)}{z_B}\prod_{j=1}^{l(B)}(q^{mndB_i}-q^{mndB_i})\\\nonumber
&=\sum_{B}\frac{\chi_{(a|b)}(C_B)}{z_B}(q^{mnd}-q^{-mnd})
\sum_{\hat{a}+\hat{b}+1=d}\chi_{(\hat{a}|\hat{b})}(C_B)(-1)^{\hat{b}}(q^{mnd})^{(\hat{a}-\hat{b})}\\\nonumber
&=(q^{mnd}-q^{-mnd})(-1)^bq^{mnd(a-b)}.
\end{align}
where we have used the orthogonal relation
\begin{align}
\sum_{B}\frac{\chi_{A^1}(C_B)\chi_{A^2}(C_B)}{z_B}=\delta_{A^1A^2}.
\end{align}
Now we can give a proof of Theorem 6.4.
\begin{proof}
Since $A=(a|b)$ is a hook partition of $d$, i.e $a+b+1=d$. So
$k_{(a|b)}=(a-b)d$. By the formula (4.4), so we get
\begin{align}
s_{(a|b)}^{*}(q,t)=\frac{(-1)^b}{d}\left(\frac{t^d-t^{-d}}{q^d-q^{-d}}\right)+\sum_{\substack{B\vdash
d\\l(B)\geq
2}}\frac{\chi_{(a|b)}(C_B)}{z_B}\prod_{j=1}^{l(B)}\frac{t^{B_j}-t^{-B_j}}{q^{B_j}-q^{-B_j}}
\end{align}
\begin{align}
W_{(a|b)}(q,t)&=q^{-mnd(a-b)}t^{-n(m-1)d}\sum_{B}\frac{\chi_{(a|b)}(C_B)}{z_B}\left(\sum_{\mu}\chi_{\mu}(C_{mB})
\chi_{\mu}(C_{(md)})q^{\frac{n}{m}k_\mu}\right.\\\nonumber
&\left.\cdot
\frac{1}{md}\frac{t^{md}-t^{-md}}{q^{md}-q^{-md}}+\sum_{\mu}\chi_{\mu}(C_{mB})\sum_{l(\nu)\geq
2}\frac{\chi_{\mu}(C_\nu)}{z_\nu}q^{\frac{n}{m}k_\mu}\prod_{j=1}^{l(\nu)}\frac{t^{\nu_j}-t^{-\nu_j}}{q^{\nu_j}-q^{-\nu_j}}
\right)
\end{align}
Using the L'H$\hat{\text{o}}$pital's rule and the formulas (6.39)
and (6.40), we obtain
\begin{align}
\lim_{t\rightarrow
1}\frac{W_{(a|b)}(q,t)}{s^*_{(a|b)}(q,t)}=\frac{\lim_{t\rightarrow
1}\frac{dW_{(a|b)}(q,t)}{dt}}{\lim_{t\rightarrow
1}\frac{ds^*_{(a|b)}(q,t)}{dt}}=\frac{(q^{mnd}-q^{-mnd})(q-q^{-1})}{(q^{md}-q^{-md})(q^{nd}-q^{-nd})}=\Delta_{T(n,m)}(q^{d})
\end{align}
\end{proof}

\begin{example}
The Alexander polynomial of torus knot $T(2,2k+1)$ is given by
\begin{align}
\Delta_{T(2,2k+1)}(q)=\frac{q^{2k+1}-q^{-2k-1}}{q-q^{-1}}.
\end{align}
According to the formulas (5.13),(5.14) and (5.15). We have
\begin{align}
\lim_{t\rightarrow
1}\frac{W_{(1)}(q,t)}{s^*_{(1)}(q,t)}&=\Delta_{T(2,2k+1)}(q)\\
\lim_{t\rightarrow
1}\frac{W_{(2)}(q,t)}{s^*_{(2)}(q,t)}&=\Delta_{T(2,2k+1)}(q^2)\\
\lim_{t\rightarrow
1}\frac{W_{(11)}(q,t)}{s^*_{(11)}(q,t)}&=\Delta_{T(2,2k+1)}(q^2)
\end{align}
\end{example}

\begin{remark}
It is easy to show that the definition of
$\Delta^{\mathcal{K}}_{A}(t)$ can not be generalized to the case of
the links like the definition of special polynomial (6.5). This is
because the following limit
\begin{align}
\lim_{t\rightarrow
1}\frac{W_{\vec{A}}(\mathcal{L};q,t)}{W_{\vec{A}}(\bigcirc^{\otimes
L};q,t)}.
\end{align}
may not exist for a general link. Considering the Hopf link
$T(2,2)$, according to the formula (5.16), we have
\begin{align}
W_{(1),(1)}(q,t)=q^2s^*_{(2)}(q,t)+q^{-2}s^*_{(11)}(q,t),
\end{align}
But the limit
\begin{align}
\lim_{t\rightarrow
1}\frac{W_{(1),(1)}(q,t)}{s_{(1)}^*(q,t)s_{(1)}^*(q,t)}
\end{align}
does not exists for general $t$ by a direct calculation.
\end{remark}

\section{Symmetries}
In this section, we give some symmetric properties of the colored
HOMFLY polynomial.
\begin{theorem}
Given a link $\mathcal{L}$ with $L$ components, and
$\vec{A}=(A^1,..,A^L)\in \mathcal{P}^L$, we have
\begin{align}
W_{\vec{A}}(\mathcal{L};-q^{-1},t)=(-1)^{\sum_{\alpha}k_{A^\alpha}}W_{\vec{A}^t}(\mathcal{L};q,t)
\end{align}
\end{theorem}
\begin{proof}
For simplicity, we just to show the proof for a given knot
$\mathcal{K}$, it is easy to write the proof for a general link
$\mathcal{L}$ similarly. Since
\begin{align}
Q_{(d)}&=\frac{1}{\alpha_d}\sum_{\pi\in
S_d}q^{l(\pi)}\widehat{\omega_\pi},\\
Q_{(d)^t}&=\frac{1}{\beta_{d}}\sum_{\pi\in
S_d}(-q)^{-l(\pi)}\widehat{\omega_\pi}.
\end{align}
 where
$\alpha_d=\beta_d|_{q\rightarrow -q^{-1}}$ by the definition in
Section 3.3. Hence
\begin{align}
Q_{(d)}=Q_{(d)^t}|_{q\rightarrow -q^{-1}}.
\end{align}
It follows that
\begin{align}
\langle\mathcal{K}\star Q_{(d^1)}Q_{(d^2)}\cdots
Q_{(d^l)}\rangle=\langle\mathcal{K}\star
Q_{(d^1)^t}Q_{(d^2)^t}\cdots Q_{(d^l)^t}\rangle_{q\rightarrow
-q^{-1}}
\end{align}

Given a partition $\lambda$ with length $l(\lambda)=l$. By the
formula (3.19), we get
\begin{align}
Q_\lambda=\det(h_{\lambda_i+j-i})=\det(Q_{(\lambda_i+j-i)})=\sum_{\tau\in
S_l}(-1)^{sign(\tau)}Q_{(\lambda_1+\tau(1)-1)}\cdots
Q_{(\lambda_l+\tau(l)-l)}.
\end{align}
\begin{align}
Q_{\lambda^t}=\det(e_{\lambda_i+j-i})=\det(Q_{(\lambda_i+j-i)^t})=\sum_{\tau\in
S_l}(-1)^{sign(\tau)}Q_{(\lambda_1+\tau(1)-1)^t}\cdots
Q_{(\lambda_l+\tau(l)-l)^t}.
\end{align}
Thus,
\begin{align}
\langle \mathcal{K}\star Q_\lambda \rangle=\langle \mathcal{K}\star
Q_{\lambda^{t}} \rangle_{q\rightarrow -q^{-1}}.
\end{align}
is a consequence of the formula (7.5).

Moreover, by the definition (3.6), $k_{\lambda}$ is an even integer
for any partition $\lambda$ and
\begin{align}
k_{\lambda}=-k_{\lambda^{t}}.
\end{align}
We obtain
\begin{align}
W_{\lambda}(\mathcal{K};q,t)=(-1)^{k_{\lambda}}W_{\lambda^t}(\mathcal{K};-q^{-1},t).
\end{align}
\end{proof}

\begin{theorem}
Given a link $\mathcal{L}$ with $L$ components, and
$\vec{A}=(A^1,..,A^L)\in \mathcal{P}^L$, we have the following
symmetry:
\begin{align}
W_{\vec{A}}(\mathcal{L};q^{-1},t)=(-1)^{\|\vec{A}\|}W_{\vec{A}^{t}}(\mathcal{L};q,t)
\end{align}
\end{theorem}

\begin{proof}
For simplicity, we just to show the proof for a given knot
$\mathcal{K}$, it is easy to write the proof for a general link
$\mathcal{L}$ similarly. Given a permutation $\pi\in S_d$, we denote
$c(\pi)$ the cycle type of $\pi$ which is a partition. It is easy to
see that the number of the components of the link
$\widehat{\omega_\pi}$ is equal to $l(c(\pi))$. Thus the number of
the components of the link $\mathcal{K}\star
\widehat{\omega_{\pi_1}}\cdots\widehat{\omega_{\pi_l}}$ is equal to
\begin{align}
L(\mathcal{K}\star
\widehat{\omega_{\pi_1}}\cdots\widehat{\omega_{\pi_l}})=\sum_{i=1}^l
l(c(\pi_i)).
\end{align}

Before to proceed, we prove the following lemma firstly.
\begin{lemma}
Given a permutation $\pi\in S_d$, we have the following identity
\begin{align}
l(\pi)+l(c(\pi))=d\mod 2
\end{align}
\end{lemma}
\begin{proof}
For every permutation $\pi\in S_d$, its length $l(\pi)$ can be
obtained by calculating the minimal number of the crossings in the
positive braid $\omega_\pi$.  When $d=2$, we have
$S_2=\{\pi_1=(1)(2), \pi_2=(12)\}$. Hence $c(\pi_1)=(11)$, and
$c(\pi_2)=(2)$. It is clear that
\begin{align}
l(\pi_1)+l(c(\pi_1))=l(\pi_2)+l(c(\pi_2))=2.
\end{align}
So Lemma 7.3 holds when $d=2$.

Now we assume Lemma 7.3 holds for $d\leq n-1$. Given a permutation
$\pi\in S_{n}$. We first consider the special case, when $\pi$ has
the cycle form: $\pi=\pi'(n)$, where $\pi'$ is a permutation in
$S_{n-1}$. It is easy to see that $l(\pi)=l(\pi')$ and
$l(c(\pi))=l(c(\pi'))+1$. By the induction hypothesis, we have
\begin{align}
l(\pi')+l(c(\pi'))=n-1\mod 2
\end{align}
Thus we get
\begin{align}
l(\pi)+l(c(\pi))=n \mod 2.
\end{align}
Thus Lemma 7.3 holds for $\pi\in S_n$ with the cycle form $\pi'(n)$,
$\pi'\in S_{n-1}$.

For the general case, we can assume $\pi$ has the cycle form
$\pi=\sigma \tau$, where $\tau$ is the cycle containing the element
$n$ as the form $(i_1\cdots i_j n)$ for $\{i_1,..i_j\}\subset
\{1,..,n-1\}$, $1\leq j\leq n-1$, and $\sigma$ is a cycle in
$S_{n-j-1}$. Hence
\begin{align}
l(c(\pi))=l(c(\sigma))+l(c(\tau)).
\end{align}
By the property of the permutation, the number of the crossings
between $\omega_\sigma$ and $\omega_\tau$ must be an even number,
thus
\begin{align}
l(\pi)=l(\sigma)+l(\tau)+\text{even number}.
\end{align}
Combining (7.17), (7.18) and the induction hypothesis, we have
\begin{align}
l(\pi)+l(c(\pi))=n\mod 2.
\end{align}
So, we finish the proof of Lemma 7.3.
\end{proof}
We now proceed to prove Theorem 7.2. By the definition of $Q_d$, we
have
\begin{align}
&\langle\mathcal{K}\star Q_{(d^1)}Q_{(d^2)}\cdots
Q_{(d^l)}\rangle\\\nonumber&=\prod_{i=1}^l\frac{1}{\alpha_{d^i}}\sum_{\pi_i\in
S_{d^{i}}}q^{\sum_{i=1}^ll(\pi_i)}\langle\mathcal{K}\star
\widehat{\omega_{\pi_1}}\cdots\widehat{\omega_{\pi_l}}\rangle\\\nonumber
&=\prod_{i=1}^l\frac{1}{\alpha_{d^i}}\sum_{\pi_i\in
S_{d^i}}\hat{p}^{\mathcal{K}\star \widehat{\omega_{\pi_1}}\cdots
\widehat{\omega_{\pi_l}}}_{2g+1-\sum_{i=1}^l
l(c(\pi_i))}(t)(q-q^{-1})^{2g-\sum_{i=1}^l
l(c(\pi_i))}q^{\sum_{i=1}^{l}l(\pi_i)}
\end{align}
\begin{align}
&\langle\mathcal{K}\star Q_{(d^1)^t}Q_{(d^2)^t}\cdots
Q_{(d^l)^t}\rangle\\\nonumber&=\prod_{i=1}^l\frac{1}{\beta_{d^i}}\sum_{\pi_i\in
S_{d^{i}}}(-q)^{-\sum_{i=1}^ll(\pi_i)}\langle\mathcal{K}\star
\widehat{\omega_{\pi_1}}\cdots\widehat{\omega_{\pi_l}}\rangle\\\nonumber
&=\prod_{i=1}^l\frac{1}{\beta_{d^i}}\sum_{\pi_i\in
S_{d^i}}\hat{p}^{\mathcal{K}\star \widehat{\omega_{\pi_1}}\cdots
\widehat{\omega_{\pi_l}}}_{2g+1-\sum_{i=1}^l
l(c(\pi_i))}(t)(q-q^{-1})^{2g-\sum_{i=1}^l
l(c(\pi_i))}(-q)^{-\sum_{i=1}^{l}l(\pi_i)}
\end{align}

By the definition of $\alpha_d$ and $\beta_d$ as showed in Section
3.3, we have
\begin{align}
\alpha_d=\beta_{d}|_{q\rightarrow -q^{-1}}=\beta_d|_{q\rightarrow
q^{-1}}.
\end{align}
Finally, according to Lemma 7.3 and the formulas (7.20), (7.21) and
(7.22), one has
\begin{align}
\langle\mathcal{K}\star Q_{(d^1)}Q_{(d^2)}\cdots
Q_{(d^l)}\rangle=(-1)^{d^1+\cdots d^l}\langle\mathcal{K}\star
Q_{(d^1)^t}Q_{(d^2)^t}\cdots Q_{(d^l)^t}\rangle_{q\rightarrow
q^{-1}}
\end{align}
By the definition of $Q_\lambda$ and $Q_{\lambda^t}$ as showed by
formulas (7.6) and (7.7), we obtain
\begin{align}
\langle \mathcal{K}\star Q_\lambda \rangle=(-1)^{|\lambda|}\langle
\mathcal{K}\star Q_{\lambda^{t}} \rangle_{q\rightarrow q^{-1}}.
\end{align}
Since $k_{\lambda}=-k_{\lambda^t}$, the identity
\begin{align}
W_{\lambda}(\mathcal{K};q,t)=(-1)^{|\lambda|}W_{\lambda^t}(\mathcal{K};q^{-1},t).
\end{align}
follows immediately.
\end{proof}

\end{document}